\documentclass[leqno,11pt]{amsart}

\usepackage[colorlinks,hypertexnames=false]{hyperref}
\hypersetup{
  allcolors=blue,
  pdftitle={The large mass limit of monopoles: abelian limits and Dirac singularities},
  pdfauthor={Daniel Fadel},
  pdfkeywords={magnetic monopoles, asymptotically conical manifolds, Dirac singularities, abelianization}
}
\usepackage{amsmath,amsthm,amssymb,mathrsfs,mathtools}
\usepackage[alphabetic]{amsrefs}
\usepackage[T1]{fontenc}
\usepackage[utf8]{inputenc}
\usepackage[english]{babel}
\usepackage{enumitem}

\numberwithin{equation}{section}

\let\originalleft\left
\let\originalright\right
\renewcommand{\left}{\mathopen{}\mathclose\bgroup\originalleft}
\renewcommand{\right}{\aftergroup\egroup\originalright}

\makeatletter
\renewcommand*{\eqref}[1]{%
    \hyperref[{#1}]{\textup{\tagform@{\ref*{#1}}}}}
\makeatother

\newtheorem{theorem}{Theorem}[section]
\newtheorem{Mtheorem}{Theorem}

\newtheorem{proposition}[theorem]{Proposition}
\newtheorem{lemma}[theorem]{Lemma}
\newtheorem{corollary}[theorem]{Corollary}
\theoremstyle{definition}

\newtheorem{problem}[theorem]{Problem}
\theoremstyle{remark}
\newtheorem{remark}[theorem]{Remark}

\newcommand{\SU}{\mathrm{SU}}
\newcommand{\U}{\mathrm U}
\newcommand{\vol}{\mathrm{vol}}
\newcommand{\Vol}{\mathrm{Vol}}
\newcommand{\supp}{\operatorname{spt}}
\newcommand{\dist}{\operatorname{dist}}
\newcommand{\tr}{\operatorname{tr}}
\newcommand{\cZ}{\mathcal Z}
\newcommand{\cS}{\mathcal S}
\newcommand{\loc}{\mathrm{loc}}

\newcommand{\norm}[1]{\left\lVert #1\right\rVert}
\newcommand{\abs}[1]{\left\lvert #1\right\rvert}
\newcommand{\ip}[2]{\left\langle #1,#2\right\rangle}

\title[The large mass limit of monopoles]
{The large mass limit of monopoles:\\
abelian limits and Dirac singularities}

\author{Daniel Fadel}
\address{Instituto de Ci\^encias Matem\'aticas e de Computa\c{c}\~ao,
Universidade de S\~ao Paulo, S\~ao Carlos, Brazil}
\email{daniel.fadel@icmc.usp.br}

\begin{document}

\begin{abstract}
Let $(A_i,\Phi_i)$ be finite energy $\SU(2)$ monopoles of charge $k>0$ on
an asymptotically conical $3$-manifold with one end, with masses
$m_i\to\infty$. After passing to a subsequence, the mass-renormalized
energy measures concentrate at finitely many points $x_a$ with concentration
weights $4\pi K_a$, where $K_a$ is the total charge of the complete finite cluster
of mass-one Euclidean monopoles lying over $x_a$. We prove that, on the complement $M$ of these points, the fields
abelianize exponentially. After translating the Higgs fields by their
masses along the unit Higgs directions and applying gauge transformations,
the translated pairs converge smoothly locally to a reducible monopole
$(A_\infty,\Phi_\infty)$ of the form
\[
    \Phi_\infty=-u\Psi_\infty,
    \qquad
    F_{A_\infty}=-*du\,\Psi_\infty,
    \qquad
    u=4\pi\sum_aK_aG(\,\cdot\,,x_a),
\]
where $\Psi_\infty$ is a parallel unit section and $G$ is the minimal
positive Green function. Consequently, $x_a$ is a Dirac singularity of
charge $K_a$. The singular part of the residual limit is determined by the
weighted $0$-cycle of concentration points and total cluster charges, and
does not retain the individual Euclidean profiles or their separation
hierarchy. We also show that $k-\sum_aK_a$ is exactly the charge
escaping through the asymptotically conical end, and describe the residual
flat abelian ambiguity.
\end{abstract}

\maketitle

{\hypersetup{hidelinks}\tableofcontents}

\section{Introduction}

This paper continues the study of the large mass limit $m_i\to\infty$ of finite energy monopoles on asymptotically conical $3$-manifolds initiated in the joint work \cite{fadel2019limit}. Throughout, a monopole means an $\SU(2)$ solution of the three-dimensional Bogomolny equation \eqref{eq:bogomolny}. The higher-dimensional $\mathrm G_2$ and Calabi--Yau monopole equations are mentioned below only for comparison.

The large mass limit has two distinguished levels of description. Near each
point of the finite concentration set $\cS$ of the sequence, moving-centre
rescalings at the natural scale $m_i^{-1}$ produce a finite cluster of
mass-one Euclidean monopoles. At the original scale, away from $\cS$, the
monopoles abelianize and, after translating the Higgs fields by their masses
along the unit Higgs directions, converge to a monopole with Dirac
singularities at the points of $\cS$.
The purpose of the present paper is to identify this residual field more
precisely: its scalar part is the Green potential of the concentration cycle,
and its Dirac charges are the total charges of the corresponding Euclidean
clusters.

\subsection{Geometric setting and concentration input}\label{subsec:context}

Let $(X^3,g)$ be a complete, connected, oriented Riemannian $3$-manifold,
and let $P\to X$ be a principal $\SU(2)$-bundle. We denote by
$\mathfrak g_P:=P\times_{\mathrm{Ad}}\mathfrak{su}(2)$ the adjoint bundle,
by $\mathscr A(P)$ the space of smooth connections on $P$, and use the
$\mathrm{Ad}$-invariant inner product
$\ip{a}{b}=-2\tr(ab)$ on $\mathfrak{su}(2)$. The Yang--Mills--Higgs
energy of a connection $A\in\mathscr A(P)$ and a Higgs field
$\Phi\in\Gamma(\mathfrak g_P)$ is
\[
    \mathscr E_X(A,\Phi)
    :=
    \frac12\int_X
    \left(\abs{F_A}^2+\abs{\nabla_A\Phi}^2\right)\vol_g,
\]
and its finite energy configuration space is
\begin{equation}\label{eq:finite-energy-configuration-space}
    \mathscr C(P)
    :=
    \left\{
       (A,\Phi)\in\mathscr A(P)\times\Gamma(\mathfrak g_P):
       \abs{F_A},\abs{\nabla_A\Phi}\in L^2(X)
    \right\}.
\end{equation}
A finite energy monopole is a configuration $(A,\Phi)\in\mathscr C(P)$
satisfying the Bogomolny equation
\begin{equation}\label{eq:bogomolny}
    F_A=*\nabla_A\Phi.
\end{equation}
By the Bianchi identity, monopoles are critical points of
$\mathscr E_X$.

We assume throughout that $(X,g)$ is asymptotically conical (AC) with one end. Thus, for some compact set $K\subset X$, a closed connected oriented
Riemannian surface $(\Sigma,g_\Sigma)$, and an orientation-preserving
diffeomorphism
\[
    \varphi:(R_0,\infty)\times\Sigma\longrightarrow X\setminus K,
\]
the metric $\varphi^*g$ is asymptotic to the cone metric
$g_C=d\rho^2+\rho^2g_\Sigma$ in the sense that
\[
    \abs{\nabla_C^j(\varphi^*g-g_C)}_{g_C}
    =O(\rho^{-\nu-j})
    \qquad\text{for every }j\geqslant0
\]
for some $\nu>0$. We extend the cone coordinate to a radius function
$\rho$ on $X$ and write $\Sigma_R:=\{\rho=R\}$ for its large level
surfaces.

By \cite{fadel2023asymptotics}*{Theorem~1.1}, every
$(A,\Phi)\in\mathscr C(P)$ determines a unique number $m\in[0,\infty)$ such that
$m-\abs\Phi\in L^6(X)$. If $(A,\Phi)$ is a monopole, then
\[
    \lim_{\rho\to\infty}\abs\Phi=m,
    \qquad
    \abs\Phi\leqslant m\quad\text{on }X,
\]
and $m$ is called its mass. When $m>0$, the charge
\[
    k
    :=
    \frac{1}{4\pi m}
    \int_X\left\langle F_A\wedge\nabla_A\Phi\right\rangle
\]
is an integer. In the monopole case, for every sufficiently large $R$,
the normalized Higgs field $\Phi/\abs\Phi$ on $\Sigma_R\cong\Sigma$ determines a homotopy
class of maps $\Sigma\to\mathbb S^2\subset\mathfrak{su}(2)$ whose Brouwer
degree is $k$. Finally, the Bogomolny decomposition gives
\begin{equation}\label{eq:general-energy-formula}
    \mathscr E_X(A,\Phi)
    =
    \pm4\pi mk
    +\frac12\norm{F_A\mp*\nabla_A\Phi}_{L^2(X)}^2.
\end{equation}
Consequently,
$\mathscr E_X(A,\Phi)\geqslant4\pi m\abs{k}$, and monopoles of mass
$m>0$ and charge $k>0$ are precisely the absolute minimizers among
configurations with the same $m$ and $k$; in particular,
$\mathscr E_X(A,\Phi)=4\pi mk$ for such a monopole.

The concentration analysis initiated in the published article
\cite{fadel2019limit} is used below in the corrected and expanded form
given in arXiv v5 of the same work
\cite{fadeloliveira2026limitv5}. Whenever the two versions differ, our
references below are to the latter. The point relevant here is that
several monopole cores at the natural scale $m_i^{-1}$ may converge to the
same point of $X$ while separating in the rescaled coordinates. The
concentration weight at $x_a$ in the limiting mass-renormalized energy measure
is therefore $4\pi K_a$, where $K_a$ is the total charge of the complete
cluster over $x_a$; it need not equal the charge of any single pointed
Euclidean monopole limit. We state the precise input needed in this paper.

\begin{theorem}[Concentration and cluster decomposition {\cite[Theorem 1.1]{fadeloliveira2026limitv5}}]
\label{thm:concentration-clusters}
Let $(A_i,\Phi_i)\in\mathscr C(P)$ be monopoles of charge $k>0$ and
masses $m_i\to\infty$. Set
\[
    e_i
    :=
    \frac12\left(
        \abs{F_{A_i}}^2+\abs{\nabla_{A_i}\Phi_i}^2
    \right)
    =
    \abs{\nabla_{A_i}\Phi_i}^2
    =
    \abs{F_{A_i}}^2,
    \qquad
    \mu_i:=m_i^{-1}e_i\vol_g.
\]
Thus $\mu_i$ is the mass-renormalized energy measure of
$(A_i,\Phi_i)$. By the energy formula
\eqref{eq:general-energy-formula},
\begin{equation}\label{eq:total-mass-mu}
    \mu_i(X)=4\pi k.
\end{equation}

After passing to a subsequence, still denoted by the same indices, suppose
that
\[
    \mu_i\rightharpoonup\mu
\]
as Radon measures. Fix a uniform geometric radius $\rho_0>0$, no larger
than the regularity radius in
\cite{fadeloliveira2026limitv5}*{Theorem~5.1}, and define, for this
subsequence,
\begin{equation}\label{eq:main-concentration-zero-sets}
\begin{split}
    \cS
    &:=
    \bigcap_{0<r\leqslant\rho_0}
    \left\{
       x\in X:
       \liminf_{i\to\infty}
       \mu_i\bigl(B_r(x)\bigr)\geqslant4\pi
    \right\},
    \\
    \cZ
    &:=
    \bigcap_{N\geqslant1}
    \overline{\bigcup_{i\geqslant N}\Phi_i^{-1}(0)}.
\end{split}
\end{equation}
Then
\[
    \cS=\cZ=\{x_1,\ldots,x_\ell\}
\]
is finite, possibly empty.

For every $a\in\{1,\ldots,\ell\}$, there exist an integer
$N_a\geqslant1$, sequences of zeros
\[
    p_{i,a,\beta}\in\Phi_i^{-1}(0),
    \qquad
    p_{i,a,\beta}\longrightarrow x_a,
    \qquad
    1\leqslant\beta\leqslant N_a,
\]
and mass-one Euclidean monopoles
$(A_{a,\beta},\Phi_{a,\beta})$ of charges
$q_{a,\beta}\in\mathbb Z_{>0}$, obtained by rescaling about these
centres at the mass scale. Distinct centres separate at that scale:
\[
    m_i\,d(p_{i,a,\beta},p_{i,a,\gamma})
    \longrightarrow\infty
    \qquad(\beta\neq\gamma).
\]
If
\[
    K_a:=\sum_{\beta=1}^{N_a}q_{a,\beta},
\]
then
\begin{equation}\label{eq:cluster-energy-intro}
    \mu(\{x_a\})
    =
    4\pi K_a
    =
    \sum_{\beta=1}^{N_a}
    \mathscr E_{\mathbb R^3}(A_{a,\beta},\Phi_{a,\beta}),
    \qquad
    1\leqslant N_a\leqslant K_a.
\end{equation}
Consequently,
\begin{equation}\label{eq:measure-intro}
    \mu_i\rightharpoonup\mu
    =
    4\pi\sum_{a=1}^{\ell}K_a\delta_{x_a},
    \qquad
    \sum_{a=1}^{\ell}K_a\leqslant k.
\end{equation}
In particular,
\[
    \supp\mu=\cS=\cZ.
\]
\end{theorem}

The stronger complete-cluster statement
\cite{fadeloliveira2026limitv5}*{Proposition~8.3} also gives a local
no-loss identity: after removing large mass-scale balls around the
moving centres, no mass-renormalized energy remains near the
concentration point in the iterated limit. Thus the profiles above
exhaust not only the total cluster charge $K_a$ but also the full
concentration weight $4\pi K_a$. We shall not need this stronger
quantitative statement below.

The possible deficit in \eqref{eq:measure-intro} is mass-renormalized energy
escaping through the AC end. We describe the complementary behaviour on
\[
    M:=X\setminus\cS.
\]
Although the individual zero sets move with $i$, the equality
$\cS=\cZ$ immediately gives the following local consequence.

\begin{lemma}[Absence of zeros on compact subsets]\label{lem:zero-free}
For every compact $K\Subset M$, one has
$K\cap\Phi_i^{-1}(0)=\varnothing$ for all sufficiently large $i$.
\end{lemma}

\begin{proof}
Otherwise, a sequence of zeros in $K$ would have an accumulation point in
$K$, contrary to \eqref{eq:main-concentration-zero-sets} and
Theorem~\ref{thm:concentration-clusters}.
\end{proof}

\subsection{Main result}\label{subsec:main_results}

Let $G$ be the minimal positive Green function of $X$. The scalar Higgs
defect has an exact Green representation, and the measure convergence
\eqref{eq:measure-intro} determines its limit.

\begin{Mtheorem}[Abelian limit and Dirac singularities]\label{thm:main}
Let $(X^3,g)$ be an oriented AC manifold with one end, let $P\to X$ be a
principal $\SU(2)$-bundle, and let
$(A_i,\Phi_i)\in\mathscr C(P)$ be monopoles of charge $k>0$ and
masses $m_i\to\infty$. Pass to a subsequence for which
Theorem~\ref{thm:concentration-clusters} holds, and write
\[
    \cS=\{x_1,\ldots,x_\ell\},
    \qquad
    M=X\setminus\cS.
\]
Then the following statements hold.
\begin{enumerate}[label=\textnormal{(\roman*)}]
\item For every compact $K\Subset M$, one has $\Phi_i\neq0$ on $K$ for
all sufficiently large $i$. On such a set define
\[
    \Psi_i:=\frac{\Phi_i}{\abs{\Phi_i}},
    \qquad
    u_i:=m_i-\abs{\Phi_i},
    \qquad
    \widetilde\Phi_i:=\Phi_i-m_i\Psi_i=-u_i\Psi_i.
\]
Then
\begin{equation}\label{eq:scalar-main}
    u_i\longrightarrow
    u:=4\pi\sum_{a=1}^{\ell}K_aG(\,\cdot\,,x_a)
    \qquad\text{in }C^\infty_{\loc}(M).
\end{equation}

\item The sequence becomes locally abelian on $M$. More precisely, for
every $K\Subset M$ and every $j\geqslant0$, there are
$c_{K,j},C_{K,j}>0$ such that
\begin{equation}\label{eq:transverse-main}
    m_i\norm{\nabla_{A_i}\Psi_i}_{C^j(K)}
    +\norm{(F_{A_i})^\perp}_{C^j(K)}
    +\norm{(\nabla_{A_i}\Phi_i)^\perp}_{C^j(K)}
    \leqslant C_{K,j}e^{-c_{K,j}m_i}.
\end{equation}
Here $\perp$ denotes projection onto
$\langle\Psi_i\rangle^\perp\subset\mathfrak g_P$.

\item After passing to a further subsequence and applying gauge
transformations along a compact exhaustion of $M$, the translated pairs
$(A_i,\widetilde\Phi_i)$ converge smoothly locally to a reducible monopole
$(A_\infty,\Phi_\infty)$ on $M$. There is a parallel unit section
$\Psi_\infty$ such that
\begin{equation}\label{eq:limit-main}
    \Phi_\infty=-u\Psi_\infty,
    \qquad
    \nabla_{A_\infty}\Psi_\infty=0,
    \qquad
    F_{A_\infty}=-*du\,\Psi_\infty.
\end{equation}

\item The limiting monopole has a Dirac singularity of charge $K_a$ at
$x_a$. More precisely, as $r_a=d(\,\cdot\,,x_a)\downarrow0$,
\begin{equation}\label{eq:dirac-main}
    u=\frac{K_a}{r_a}+O(1),
    \qquad
    \abs{F_{A_\infty}}=\frac{K_a}{r_a^2}+O(1).
\end{equation}
Its scalar longitudinal curvature
$\ip{F_{A_\infty}}{\Psi_\infty}=-*du$ is uniquely determined by the
weighted $0$-cycle $D=\sum_aK_ax_a$. Once the parallel
$\U(1)$-reduction and its unit generator are fixed, the limiting
connection is determined, modulo $\U(1)$ gauge, up to addition of a
closed abelian $1$-form.
\end{enumerate}
\end{Mtheorem}

\begin{remark}[Macroscopic and microscopic data]
\label{rmk:cluster-blindness}
At a concentration point $x_a$, a choice of complete moving-centre
extraction produces individual charges $q_{a,\beta}$ with
$K_a=\sum_\beta q_{a,\beta}$. For a fixed subsequential limit of the
mass-renormalized energy measures, the integer $K_a$ and the concentration weight $4\pi K_a$
are intrinsic, whereas the number of profiles, their individual charges,
and the limiting Euclidean monopoles may depend on a further subsequence
and on the choice of moving centres; see
\cite{fadeloliveira2026limitv5}*{Remark~8.3}. The Green potential and
the singular part of the residual field depend on the cluster only through
$K_a$. Thus the residual singular data record the weighted $0$-cycle
$D=\sum_aK_ax_a$, but not the number or charges of the individual
Euclidean profiles, their relative phases, or the hierarchy of scales at
which their centres separate.
\end{remark}

Taken together, Theorems~\ref{thm:concentration-clusters} and
\ref{thm:main} give a precise description of the degeneration at finite
points at two distinguished levels. The Euclidean clusters are the
microscopic nonabelian data, whereas the singular reducible monopole is the
macroscopic residual field.
These are the data at finite points expected in a compactification picture;
it does not by itself construct a topology on a compactified moduli space,
and it leaves open the finer description of clusters escaping through the
end. This interpretation is consistent with the picture from gauge theory
suggested by Donaldson and Segal \cite{donaldson2009gauge}.

There are related higher-dimensional developments. Fadel--Nagy--Oliveira
established the AC asymptotic theory for $\mathrm G_2$-monopoles
\cite{fadel2020asymptotic}*{Main Theorems~1 and~2}.
Parise--Pigati--Stern constructed current limits of codimension three for
Yang--Mills--Higgs fields
\cite{parise2025nonabelian}*{Theorem~1.2 and Corollary~1.8}, while Li
proved calibrated concentration and singular abelian limits for large mass
$\mathrm G_2$ and Calabi--Yau monopoles
\cite{li2025large}*{Theorems~1.12 and~4.3}. In those settings the weighted $0$-cycle is replaced by a calibrated cycle of codimension three.

The reverse construction on AC $3$-manifolds was carried out by Oliveira.
When $b_2(X)=0$, \cite{oliveira2016monopoles}*{Theorem~1} constructs large
mass monopoles by inserting separated charge-one BPS monopoles into a
reducible Dirac background. The parameters include the core positions,
relative phases, and a flat abelian class in $H^1(X;S^1)$; the last is the
ambiguity appearing in the limiting connection below. The examples in
\cite{fadel2019limit}*{Section~3} also exhibit Euclidean monopoles of
higher charge, as well as charge escaping through the AC end. Habibi Esfahani
proved a complementary gluing theorem for monopoles with prescribed Dirac
singularities on closed rational homology $3$-spheres
\cite{esfahani2022singular}*{Theorem~1}; see also
\cite{esfahani2022thesis}*{Chapter~1}.

\subsection{Organization}\label{subsec:organization}

Section~\ref{sec:splitting} recalls the Green representation, introduces
the splitting in the large Higgs field region, and proves local clearing of
the mass-renormalized energy. Section~\ref{sec:potential} identifies the limiting Green
potential and the longitudinal curvature. Section~\ref{sec:abelian}
proves exponential transverse decay and gauge compactness.
Section~\ref{sec:singularities} identifies the Dirac charges and the charge
escaping through the AC end. Section~\ref{sec:further-questions} records
several directions suggested by these two distinguished levels of the
large mass limit.

\subsection{Notation and conventions}\label{subsec:notation}

Our notation and normalization follow the principal bundle conventions of
\cite{fadel2023asymptotics}. The adjoint bundle is denoted by
$\mathfrak g_P$, and $\mathscr C(P)$ always denotes
\eqref{eq:finite-energy-configuration-space}. For
$A\in\mathscr A(P)$, the symbol $\nabla_A$ denotes the induced covariant
derivative, coupled with the Levi--Civita connection when acting on
$\mathfrak g_P$-valued tensors. We reserve $\abs{\,\cdot\,}$ for
pointwise norms and $\norm{\,\cdot\,}$ for norms on spaces of sections.

On the complement of a Higgs zero set, we write
\[
    \Psi:=\frac{\Phi}{\abs\Phi},
    \qquad
    \mathfrak g_P
    =\langle\Psi\rangle\oplus\langle\Psi\rangle^\perp,
\]
and denote the corresponding components by $\parallel$ and $\perp$.
For a $\mathfrak g_P$-valued tensor $T$ and a compact set $K$, we use
\[
    \norm{T}_{C^j(K)}
    :=
    \sum_{q=0}^j\sup_K\abs{\nabla_A^qT},
\]
where the covariant derivative is coupled with the Levi--Civita connection
on tensor indices. We write $B_r(x)$ for the geodesic ball of radius $r$
centred at $x$. Convergence of Radon measures is always weak-$*$ convergence against
functions in $C_c^0(X)$. The Riemannian volume measure is denoted by
$\vol_g$ and is omitted under integral signs unless useful. Energy quantities
divided by the monopole mass are called \emph{mass-renormalized}, whereas a
Higgs field divided by its mass is called \emph{mass-normalized}. We reserve
the phrase \emph{normalized Higgs field} for the unit-length field
$\Phi/\abs\Phi$. Thus $k$, $K_a$, and $q_{a,\beta}$ always denote the
charges themselves; no normalization by the mass is involved. We use the
geometer's convention $\Delta=d^*d$ on functions. Constants denoted by
$c,C>0$ may change from line to line and are independent of the mass and
the sequence index unless otherwise stated.

\section{Green representation and local clearing}\label{sec:splitting}

We now turn to the residual fields on $M=X\setminus\cS$. The first
ingredient is the exact Green representation of the scalar Higgs defect.

\subsection{The Green representation}\label{subsec:green}

By \cite{fadel2023asymptotics}*{Corollary~2.9}, the AC manifold $X$ is
nonparabolic and has a minimal positive Green function $G$. With the
positive Laplacian $\Delta=d^*d$, it satisfies
$\Delta_xG(x,y)=\delta_y$ and
\begin{equation}\label{eq:Green-estimates}
    G(x,y)=\frac1{4\pi d(x,y)}+O(1)\quad (x\to y),
    \qquad
    G(x,y)=O\bigl(d(x,y)^{-1}\bigr).
\end{equation}
The weighted AC Laplacian theory recalled in
\cite{fadel2023asymptotics}*{Theorem~2.12(iv)} gives, for $x$ in a fixed compact
set and $\rho(y)\to\infty$,
\begin{equation}\label{eq:Green-infinity}
    G(x,y)=\frac{1}{\Vol(\Sigma)}\rho(y)^{-1}
    +O(\rho(y)^{-1-\eta})
\end{equation}
for some $\eta>0$, with the corresponding differentiated estimates. The
minimal Green function is symmetric, so the same expansion applies when
the first variable tends to infinity and the second remains in a fixed
compact set.

\begin{theorem}[Green representation
\cite{fadel2023asymptotics}*{Theorem~3.11}]
\label{thm:Green-representation}
For every monopole $(A,\Phi)\in\mathscr{C}(P)$ of mass $m$,
\begin{equation}\label{eq:Green-representation}
    m^2-\abs{\Phi}^2
    =2\int_XG(\,\cdot\,,y)\abs{\nabla_A\Phi}^2(y).
\end{equation}
\end{theorem}

\subsection{Splitting in the large Higgs field region}

Fix $K\Subset M$. Lemma~\ref{lem:zero-free} shows that the following
quantities are defined on $K$ for all sufficiently large $i$:
\[
    \Psi_i:=\frac{\Phi_i}{\abs{\Phi_i}},
    \qquad u_i:=m_i-\abs{\Phi_i},
    \qquad \widetilde\Phi_i:=-u_i\Psi_i.
\]
The unit Higgs field induces the orthogonal decomposition
\[
    \mathfrak g_P
    =
    \langle\Psi_i\rangle\oplus\langle\Psi_i\rangle^\perp.
\]
Accordingly, every $\mathfrak g_P$-valued form $\xi$ decomposes as
\[
    \xi=\xi^\parallel+\xi^\perp,
    \qquad
    \xi^\parallel=\ip{\xi}{\Psi_i}\Psi_i.
\]
Differentiating $\Phi_i=(m_i-u_i)\Psi_i$ and using the monopole equation
gives
\begin{equation}\label{eq:splitting-identities}
\begin{split}
    \nabla_{A_i}\Phi_i
        &=-du_i\otimes\Psi_i+(m_i-u_i)\nabla_{A_i}\Psi_i,\\
    (\nabla_{A_i}\Phi_i)^\perp
        &=\abs{\Phi_i}\nabla_{A_i}\Psi_i,\\
    \widehat\omega_i:=\ip{F_{A_i}}{\Psi_i}
        &=*d\abs{\Phi_i}=-*du_i,\\
    *F_{A_i}-\nabla_{A_i}\widetilde\Phi_i
        &=m_i\nabla_{A_i}\Psi_i.
\end{split}
\end{equation}
Since $\nabla_{A_i}^*\nabla_{A_i}\Phi_i=0$, the scalar Kato identity is
\begin{equation}\label{eq:scalar-u}
    \Delta u_i
    =
    \frac{\abs{(\nabla_{A_i}\Phi_i)^\perp}^2}{\abs{\Phi_i}}.
\end{equation}

We first show that the mass-renormalized energy density vanishes on compact subsets of $M$.
The proof uses mass-scale estimates from the corrected and expanded
arXiv version \cite{fadeloliveira2026limitv5} only to enter a regime
in which the favourable Higgs term in the Bochner formula is coercive;
the final mean value argument takes place on a fixed geometric scale.

\begin{proposition}[Local clearing of the mass-renormalized energy]\label{prop:local-clearing}
For every compact $K\Subset M$,
\begin{equation}\label{eq:local-clearing}
    \norm{m_i^{-1}e_i}_{C^0(K)}\longrightarrow0.
\end{equation}
\end{proposition}

\begin{proof}
Choose nested relatively compact open sets
\[
    K\Subset U_0\Subset U_1\Subset V\Subset M.
\]
Since $\overline V\Subset M$ and the limiting measure in
\eqref{eq:measure-intro} is supported on $\cS$, one has
$\mu(\overline V)=0$. The Portmanteau theorem therefore gives
\begin{equation}\label{eq:mu-clears-V}
    0\leq\delta_i:=\mu_i(V)
    \leq\mu_i(\overline V)\longrightarrow0.
\end{equation}

Set $\Lambda:=4\pi k$. Since the $(A_i,\Phi_i)$ are monopoles of fixed
charge $k$,
\[
    m_i^{-1}\norm{\nabla_{A_i}\Phi_i}_{L^2(X)}^2
    =m_i^{-1}\mathscr E_X(A_i,\Phi_i)
    =4\pi k=\Lambda.
\]
Let $R:=R_\Lambda$ and $\varepsilon_\Lambda$ be the constants in
\cite{fadeloliveira2026limitv5}*{Corollary~6.1}, and let $m_0$ and
$C_R$ denote the corresponding constants from the mass-uniform
$\varepsilon$-regularity estimate
\cite{fadeloliveira2026limitv5}*{Theorem~5.1}. Since
$\delta_i\to0$ and $m_i\to\infty$, after increasing $i$ we may assume
\[
    \delta_i<\varepsilon_\Lambda,
    \qquad
    m_i>m_0,
    \qquad
    Rm_i^{-1}\leq\rho_0.
\]
For all sufficiently large $i$, every ball $B_{R/m_i}(x)$ with
$x\in U_1$ is contained in $V$ and satisfies
\[
    m_i^{-1}\mathscr E_{B_{R/m_i}(x)}(A_i,\Phi_i)
    =\mu_i\bigl(B_{R/m_i}(x)\bigr)
    \leq\delta_i<\varepsilon_\Lambda.
\]
Applying
\cite{fadeloliveira2026limitv5}*{Corollary~6.1} and the
$\varepsilon$-regularity estimate
\cite{fadeloliveira2026limitv5}*{Theorem~5.1} at radius $R/m_i$ yields,
uniformly for $x\in U_1$,
\begin{equation}\label{eq:mass-scale-entry}
    e_i(x)\leq \widetilde C_R\delta_i m_i^4,
    \qquad
    \abs{\Phi_i}(x)>\frac{m_i}{4},
    \qquad
    \widetilde C_R:=C_RR^{-3}.
\end{equation}

We next improve the first estimate from the mass scale to a fixed scale.
The inequalities \cite{fadel2023asymptotics}*{Lemma~4.8, (4.14)--(4.15)},
together with $F_{A_i}=*\nabla_{A_i}\Phi_i$, imply on $U_1$
\begin{equation}\label{eq:clearing-bochner}
\begin{split}
    \frac12\Delta e_i
    &+\abs{\nabla_{A_i}^2\Phi_i}^2
     +c\abs{\Phi_i}^2
       \abs{(\nabla_{A_i}\Phi_i)^\perp}^2\\
    &\leq C_{U_1}e_i
     +C e_i^{1/2}
       \abs{(\nabla_{A_i}\Phi_i)^\perp}^2.
\end{split}
\end{equation}
Indeed, all the terms which are not bounded by $C_{U_1}e_i$ contain at
least two transverse factors, and the remaining factor is bounded by
$Ce_i^{1/2}$. By \eqref{eq:mass-scale-entry},
$e_i^{1/2}\leq \widetilde C_R^{1/2}\delta_i^{1/2}m_i^2$, whereas
$\abs{\Phi_i}^2\geq m_i^2/16$. Hence, the final term in
\eqref{eq:clearing-bochner} is absorbed by the coercive Higgs term for all
large $i$. We obtain
\begin{equation}\label{eq:fixed-scale-subsolution}
    \Delta e_i\leq C_{U_1}e_i
    \qquad\text{on }U_1,
\end{equation}
with a constant independent of $i$.

Choose $r>0$ so that $B_{2r}(x)\subset U_1$ for every $x\in U_0$. The
local mean value inequality applied to
\eqref{eq:fixed-scale-subsolution} gives
\[
    \sup_{B_r(x)}e_i
    \leq C_r\int_{B_{2r}(x)}e_i.
\]
A finite covering of $K$ therefore yields
\[
    \sup_Km_i^{-1}e_i
    \leq C_{K,U_1}m_i^{-1}\int_{U_1}e_i
    =C_{K,U_1}\mu_i(U_1)
    \longrightarrow0
\]
by \eqref{eq:mu-clears-V}. This is \eqref{eq:local-clearing}.
\end{proof}

\section{\texorpdfstring{The Green potential of the concentration $0$-cycle}{The Green potential of the concentration 0-cycle}}\label{sec:potential}

Define
\begin{equation}\label{eq:def-v-i}
    v_i:=m_i^{-1}\bigl(m_i^2-\abs{\Phi_i}^2\bigr)
    =u_i\left(1+\frac{\abs{\Phi_i}}{m_i}\right).
\end{equation}
By Theorem~\ref{thm:Green-representation},
\begin{equation}\label{eq:v-potential}
    v_i(x)=2\int_XG(x,y)\,d\mu_i(y),
    \qquad
    \Delta v_i=2m_i^{-1}e_i.
\end{equation}

\begin{proposition}[Convergence of the Green potentials]
\label{prop:potential-convergence}
For every $0<\alpha<1$, on $M=X\setminus\cS$,
\begin{equation}\label{eq:v-limit}
    v_i\longrightarrow v:=8\pi\sum_{a=1}^{\ell}K_aG(\,\cdot\,,x_a)
    \qquad\text{in }C^{1,\alpha}_{\loc}(M).
\end{equation}
Consequently,
\begin{equation}\label{eq:u-limit}
    u_i\longrightarrow u:=\frac12v
    =4\pi\sum_{a=1}^{\ell}K_aG(\,\cdot\,,x_a)
    \qquad\text{in }C^{1,\alpha}_{\loc}(M),
\end{equation}
and, for every $K\Subset M$,
\begin{equation}\label{eq:large-Higgs}
    \inf_K\frac{\abs{\Phi_i}}{m_i}\longrightarrow1.
\end{equation}
\end{proposition}

\begin{proof}
Fix $K\Subset M$ and choose $K\Subset U\Subset M$. Let $B_a$ be mutually
disjoint geodesic balls about $x_a$, with closures disjoint from
$\overline U$, and choose a compact set $L\Subset X$ containing
$\overline U\cup\bigcup_a\overline B_a$. We decompose
\eqref{eq:v-potential} into the integrals over $\bigcup_aB_a$, over
$L\setminus\bigcup_aB_a$, and over $X\setminus L$.

On $K\times\bigcup_a\overline B_a$, the functions
$\nabla_x^qG(x,y)$ are smooth for every $q$. Choosing the radii so that
$\mu(\partial B_a)=0$, weak convergence of the measures gives
\begin{equation}\label{eq:source-ball-limit}
    2\int_{\bigcup_aB_a}G(x,y)\,d\mu_i(y)
    \longrightarrow
    8\pi\sum_aK_aG(x,x_a)
\end{equation}
uniformly with every $x$-derivative on $K$. Indeed, the continuous map
\[
    K\ni x\longmapsto\nabla_x^qG(x,\cdot)
    \in C^0\left(\bigcup_a\overline B_a\right)
\]
has compact image.

The limiting measure vanishes on the compact set
$L\setminus\bigcup_aB_a$. Away from a small neighbourhood of the
diagonal in $U\times U$, the kernel and its first $x$-derivative are
bounded and continuous, so weak convergence makes the corresponding
integrals tend uniformly to zero. Near the diagonal, Proposition~\ref{prop:local-clearing} gives
\[
    \norm{m_i^{-1}e_i}_{C^0(U)}\longrightarrow0.
\]
The local Green estimates
\[
    G(x,y)\leq C d(x,y)^{-1},
    \qquad
    \abs{\nabla_xG(x,y)}\leq C d(x,y)^{-2}
\]
therefore imply, uniformly for $x\in K$ and sufficiently small fixed
$\delta>0$,
\begin{align*}
    \int_{U\cap B_\delta(x)}
    G(x,y)m_i^{-1}e_i(y)\,dy
      &\leq C\delta^2\norm{m_i^{-1}e_i}_{C^0(U)},\\
    \int_{U\cap B_\delta(x)}
    \abs{\nabla_xG(x,y)}m_i^{-1}e_i(y)\,dy
      &\leq C\delta\norm{m_i^{-1}e_i}_{C^0(U)}.
\end{align*}
Thus the compact contribution away from the source balls converges to zero
in $C^1(K)$.

Finally, the total measures satisfy $\mu_i(X)=4\pi k$. By the
differentiated asymptotic estimate in \eqref{eq:Green-infinity},
\[
\begin{split}
    &\sup_{x\in K}
    \int_{X\setminus L}
    \left(G(x,y)+\abs{\nabla_xG(x,y)}\right)\,d\mu_i(y)\\
    &\qquad\leq
    4\pi k
    \sup_{\substack{x\in K\\y\in X\setminus L}}
    \left(G(x,y)+\abs{\nabla_xG(x,y)}\right)
    \longrightarrow0
\end{split}
\]
as $L$ exhausts $X$. Together with \eqref{eq:source-ball-limit}, this
proves $v_i\to v$ in $C^1(K)$.

To obtain the stated H\"older convergence, let
$K\Subset U_0\Subset U\Subset M$. Applying the preceding $C^1$
argument with $\overline U$ in place of $K$ gives, after enlarging the
outer relatively compact neighbourhood used there,
$v_i\to v$ in $L^p(U)$ for every finite $p$. Since $v$ is harmonic on
$U$ and
$\Delta v_i=2m_i^{-1}e_i$, the interior $W^{2,p}$ estimate gives
\[
    \norm{v_i-v}_{W^{2,p}(U_0)}
    \leq C_{p,U_0,U}
       \left(
\norm{v_i-v}_{L^p(U)}
             +\norm{m_i^{-1}e_i}_{L^p(U)}\right)
    \longrightarrow0
\]
for every finite $p$. Taking $p>3/(1-\alpha)$ gives
\eqref{eq:v-limit} by Sobolev embedding.

Because $0\leq\abs{\Phi_i}\leq m_i$, one has $0\leq u_i\leq m_i$, and
\eqref{eq:def-v-i} selects the branch
\begin{equation}\label{eq:u-from-v}
    u_i
    =m_i\left(1-\sqrt{1-\frac{v_i}{m_i}}\right)
    =\frac{v_i}{1+\sqrt{1-v_i/m_i}}.
\end{equation}
The local boundedness and $C^{1,\alpha}$ convergence of $v_i$, together
with $m_i\to\infty$, now give $u_i\to v/2$ in $C^{1,\alpha}(K)$. This
proves \eqref{eq:u-limit}; since
$\abs{\Phi_i}/m_i=1-u_i/m_i$, it also proves
\eqref{eq:large-Higgs}.
\end{proof}

\begin{corollary}[Longitudinal convergence]\label{cor:longitudinal}
For every $K\Subset M$ and $0<\alpha<1$,
\begin{equation}\label{eq:longitudinal-convergence}
    \widehat\omega_i=-*du_i\longrightarrow-*du
    \qquad\text{in }C^{0,\alpha}(K).
\end{equation}
In particular, the longitudinal curvature is determined by the concentration $0$-cycle and requires no harmonic correction.
\end{corollary}

\begin{proof}
This is the third identity in \eqref{eq:splitting-identities} combined
with \eqref{eq:u-limit}. The absence of a harmonic correction is a direct
consequence of the exact Green representation: the scalar $u$, and hence
$-*du$, is fixed by the limiting measure rather than only by its exterior
derivative.
\end{proof}

\section{Transverse decay and gauge compactness}\label{sec:abelian}

We next prove that the components orthogonal to the Higgs field vanish on
compact subsets of $M$. This is the three-dimensional form of
abelianization.

\begin{lemma}[Derivative bounds at the mass scale]\label{lem:polynomial-bounds}
For every $K\Subset M$ and every integer $q\geq0$, there is
$C_{K,q}<\infty$ such that, for all sufficiently large $i$,
\begin{equation}\label{eq:polynomial-bounds}
    \norm{\nabla_{A_i}^{q}F_{A_i}}_{C^0(K)}
    +\norm{\nabla_{A_i}^{q+1}\Phi_i}_{C^0(K)}
    \leq C_{K,q}m_i^{q+2}.
\end{equation}
In particular, on compact subsets of $M$, the covariant derivatives of
the algebraic tensors built from $F_{A_i}$, $\nabla_{A_i}\Phi_i$, and the
mass-normalized Higgs field $m_i^{-1}\Phi_i$ that occur below are bounded
by fixed powers of $m_i$. The same is true for $\Psi_i$ and for the
parallel and perpendicular projections determined by $\Psi_i$.
\end{lemma}

\begin{proof}
Choose $K\Subset U\Subset M$. As in the first step of the proof of
Proposition~\ref{prop:local-clearing}, every ball at the mass scale
$B_{R/m_i}(x)$, $x\in U$, has mass-renormalized energy below the
small energy threshold for all large $i$. Rescale such a ball by $m_i$, write
$\phi_i:=m_i^{-1}\Phi_i$ for the mass-normalized Higgs field, and use the
metric $m_i^2g$. The result is a mass-one monopole on a ball of
fixed radius, with uniformly bounded geometry; its energy on the rescaled
ball is the corresponding mass-renormalized energy of the original pair
and is therefore uniformly small. The mass-uniform
$\varepsilon$-regularity estimate gives a uniform $C^0$ bound for the
rescaled energy density. By Uhlenbeck's local gauge theorem
\cite{uhlenbeck1982connections}*{Theorem~1.3}, put the rescaled connection in Coulomb
gauge, writing it as $d+a_i$. Then
\[
    d^*a_i=0,
    \qquad
    da_i=*\bigl(d\phi_i+[a_i,\phi_i]\bigr)-a_i\wedge a_i,
    \qquad
    \nabla_{d+a_i}^*\nabla_{d+a_i}\phi_i=0.
\]
The Hodge estimate for $a_i$ and the interior elliptic estimate for
$\phi_i$, applied alternately and then differentiated, give uniform
$C^q$ bounds on a smaller fixed ball for every $q$. Scaling back gives
\eqref{eq:polynomial-bounds}. The algebraic-tensor assertion follows by
the product rule and $\abs{\Phi_i}\leq m_i$. Finally, on compact subsets
of $M$ one has $\abs{\Phi_i}\geq m_i/2$ for all large $i$ by
\eqref{eq:large-Higgs}; differentiating
$\Psi_i=\Phi_i/\abs{\Phi_i}$ then gives polynomial bounds for
$\Psi_i$ and for the associated projections.
\end{proof}

\begin{lemma}[Coercive transverse inequality]\label{lem:transverse-Bochner}
Let $U\Subset M$. For all sufficiently large $i$, the transverse field
\[
    f_i:=[\Phi_i,\nabla_{A_i}\Phi_i]
\]
satisfies
\begin{equation}\label{eq:bochner-transverse}
    \frac12\Delta\abs{f_i}^2
    +c_0\abs{\Phi_i}^2\abs{f_i}^2
    +\abs{\nabla_{A_i}f_i}^2
    \leq C_U\bigl(1+\abs{du_i}\bigr)\abs{f_i}^2
\end{equation}
on $U$. Here $c_0>0$ is universal and $C_U$ depends only on the geometry
of $U$.
\end{lemma}

\begin{proof}
Apply \cite{fadel2023asymptotics}*{Lemma~4.8, (4.16)}. For a monopole,
$[F_{A_i},\Phi_i]=*[\nabla_{A_i}\Phi_i,\Phi_i]=-*f_i$, while the parallel
parts of $F_{A_i}$ and $\nabla_{A_i}\Phi_i$ both have norm $\abs{du_i}$ by
\eqref{eq:splitting-identities}. Every mixed term on the right hand side
of that inequality is therefore bounded by
$C_U(1+\abs{du_i})\abs{f_i}^2$. This gives
\eqref{eq:bochner-transverse}.
\end{proof}

\begin{lemma}[Interior massive decay]\label{lem:massive-decay}
Let $K\Subset U$ be relatively compact subsets of a Riemannian manifold
of bounded geometry. There are $c,C>0$ such that, whenever $\lambda\geq1$
and $h\in C^2(U)$ is nonnegative and satisfies
\[
    \Delta h\leq-\lambda^2h
    \qquad\text{on }U,
\]
one has
\begin{equation}\label{eq:massive-decay-abstract}
    \sup_Kh\leq Ce^{-c\lambda}\sup_Uh.
\end{equation}
\end{lemma}

\begin{proof}
Choose a smooth domain $\Omega$ with
$K\Subset\Omega\Subset U$, and let $w_\lambda$ solve
\[
    (\Delta+\lambda^2)w_\lambda=0\quad\text{on }\Omega,
    \qquad w_\lambda=1\quad\text{on }\partial\Omega.
\]
The maximum principle applied to
$h-(\sup_\Omega h)w_\lambda$ gives
$h\leq(\sup_\Omega h)w_\lambda$. Comparison on balls in normal
coordinates with the corresponding positive radial solution of the
constant-coefficient Euclidean equation, followed by a finite chain of
balls from $\partial\Omega$ to $K$, gives
\[
    \sup_Kw_\lambda
    \leq C\exp\bigl(-c\lambda\dist(K,\partial\Omega)\bigr).
\]
The uniform geometry assumptions make the constants uniform. Absorbing
the fixed distance into $c$ proves
\eqref{eq:massive-decay-abstract}.
\end{proof}

\begin{proposition}[Exponential transverse decay]\label{prop:transverse-decay}
For every $K\Subset U\Subset M$ and every $j\geq0$, there exist
$c_{K,j},C_{K,j}>0$ such that
\begin{equation}\label{eq:exp-decay}
\begin{split}
    &\norm{(F_{A_i})^\perp}_{C^j(K)}
    +\norm{(\nabla_{A_i}\Phi_i)^\perp}_{C^j(K)}
    +m_i\norm{\nabla_{A_i}\Psi_i}_{C^j(K)}\\
    &\hspace{45mm}\leq C_{K,j}e^{-c_{K,j}m_i}
\end{split}
\end{equation}
for all sufficiently large $i$.
\end{proposition}

\begin{proof}
Choose $K\Subset U_0\Subset U$. Proposition~\ref{prop:potential-convergence}
gives
\[
    \abs{\Phi_i}\geq\frac{m_i}{2},
    \qquad
    \norm{du_i}_{C^0(U)}\leq C_U
\]
for all large $i$. Dropping the nonnegative derivative term in
\eqref{eq:bochner-transverse} and increasing $i$ once more, we obtain
\[
    \Delta\abs{f_i}^2
    \leq-cm_i^2\abs{f_i}^2
    \qquad\text{on }U.
\]
Lemma~\ref{lem:massive-decay}, followed by
Lemma~\ref{lem:polynomial-bounds}, yields
\begin{equation}\label{eq:f-C0-exp}
    \norm{f_i}_{C^0(U_0)}
    \leq C e^{-cm_i};
\end{equation}
indeed, the a priori factor $\sup_U\abs{f_i}$ is polynomial in $m_i$ and
is absorbed by decreasing $c$.

For $\mathfrak{su}(2)$ with the normalization used here,
\begin{equation}\label{eq:su2-transverse-algebra}
    \abs{[\Phi_i,\xi]}
    =\abs{\Phi_i}\abs{\xi^\perp}.
\end{equation}
Consequently, \eqref{eq:f-C0-exp} and
$\abs{\Phi_i}\geq m_i/2$ give the $C^0$ estimate for
$(\nabla_{A_i}\Phi_i)^\perp$; the monopole equation gives the same
estimate for $(F_{A_i})^\perp$, and
$(\nabla_{A_i}\Phi_i)^\perp=\abs{\Phi_i}\nabla_{A_i}\Psi_i$ gives the
claimed estimate for $m_i\nabla_{A_i}\Psi_i$.

It remains to obtain derivatives without losing the exponential rate.
For any fixed $j$, Lemma~\ref{lem:polynomial-bounds} and the product rule
give polynomial $C^{2j+2}(U_0)$ bounds for each of the three tensors in
\eqref{eq:exp-decay}. In the Coulomb gauges at the mass scale used in the proof
of that lemma, the local interpolation inequality, applied on
$K\Subset U_0$, has the form
\[
    \norm{T}_{C^j(K)}
    \leq C
       \norm{T}_{C^0(U_0)}^{\theta}
       \norm{T}_{C^{2j+2}(U_0)}^{1-\theta}
       +C\norm{T}_{C^0(U_0)}
\]
for some $\theta=\theta(j)>0$. Combining the exponential $C^0$ estimate
with the polynomial bound for higher derivatives, and decreasing the
exponent if necessary, proves \eqref{eq:exp-decay} for every $j$.
\end{proof}

\begin{corollary}[Smooth scalar and longitudinal convergence]
\label{cor:smooth-potential}
The convergences in \eqref{eq:v-limit}, \eqref{eq:u-limit}, and
\eqref{eq:longitudinal-convergence} hold in $C^\infty_{\loc}(M)$.
\end{corollary}

\begin{proof}
Fix $K\Subset U_0\Subset U\Subset M$. The limit $u$ is harmonic on
$U$, whereas \eqref{eq:scalar-u} gives
\[
    \Delta(u_i-u)
    =r_i,
    \qquad
    r_i:=\frac{\abs{(\nabla_{A_i}\Phi_i)^\perp}^2}{\abs{\Phi_i}}.
\]
Proposition~\ref{prop:transverse-decay}, applied with
$\overline U_0\Subset U$, and $\abs{\Phi_i}\geq m_i/2$ show that $r_i\to0$ in
$C^q(U_0)$ for every $q$. The interior Schauder estimate therefore gives,
for every $q\geq0$ and $0<\alpha<1$,
\[
    \norm{u_i-u}_{C^{q+2,\alpha}(K)}
    \leq C
       \left(
\norm{u_i-u}_{C^0(U_0)}
             +\norm{r_i}_{C^{q,\alpha}(U_0)}\right)
    \longrightarrow0,
\]
where the $C^0$ convergence follows from
Proposition~\ref{prop:potential-convergence}. This proves smooth
convergence of $u_i$. Formula \eqref{eq:u-from-v} then gives smooth
convergence of $v_i$, and
$\widehat\omega_i=-*du_i$ gives smooth longitudinal convergence.
\end{proof}

\begin{corollary}[Original-scale local convergence]\label{cor:original-scale}
For every $K\Subset M$ and $j\geq0$,
\begin{equation}\label{eq:full-fields-converge}
\begin{split}
    F_{A_i}+*du_i\,\Psi_i&\longrightarrow0,\\
    \nabla_{A_i}\Phi_i+du_i\otimes\Psi_i&\longrightarrow0
\end{split}
\qquad\text{in }C^j(K).
\end{equation}
In particular, $F_{A_i}$ and $\nabla_{A_i}\Phi_i$ are uniformly bounded
in every $C^j(K)$.
\end{corollary}

\begin{proof}
The orthogonal decompositions are
\[
    \nabla_{A_i}\Phi_i
       =-du_i\otimes\Psi_i+(\nabla_{A_i}\Phi_i)^\perp,
    \qquad
    F_{A_i}
       =-*du_i\,\Psi_i+(F_{A_i})^\perp.
\]
The longitudinal factors $du_i$ converge smoothly by
Corollary~\ref{cor:smooth-potential}, while the transverse remainders tend
to zero in every $C^j$ by Proposition~\ref{prop:transverse-decay}. This
proves \eqref{eq:full-fields-converge} and the uniform local bounds.
\end{proof}

\begin{corollary}[Local energy convergence]\label{cor:local-energy}
On $M$ one has
\begin{equation}\label{eq:local-energy-convergence}
    \abs{F_{A_i}}^2\longrightarrow\abs{du}^2,
    \qquad
    \abs{\nabla_{A_i}\Phi_i}^2\longrightarrow\abs{du}^2
    \quad\text{in }C^\infty_{\loc}(M).
\end{equation}
In particular, no further bubbling occurs on compact subsets of
$M$ at the original scale; all bubbling at finite points is confined to $\cS$.
\end{corollary}

\begin{proof}
The parallel and perpendicular summands are orthogonal, so
\[
    \abs{F_{A_i}}^2
       =\abs{du_i}^2+\abs{(F_{A_i})^\perp}^2,
    \qquad
    \abs{\nabla_{A_i}\Phi_i}^2
       =\abs{du_i}^2
        +\abs{(\nabla_{A_i}\Phi_i)^\perp}^2.
\]
Corollary~\ref{cor:smooth-potential} and
Proposition~\ref{prop:transverse-decay} give
\eqref{eq:local-energy-convergence}. The uniform curvature bounds in
Corollary~\ref{cor:original-scale} exclude any further bubbling on a
compact subset of $M$.
\end{proof}

\begin{proposition}[Gauge compactness on the punctured manifold]
\label{prop:gauge-compactness}
There are a subsequence, a compact exhaustion $K_1\Subset K_2\Subset
\cdots\Subset M$, integers $i(j)$, gauge transformations $g_{i,j}$ on
$K_j$ for $i\geq i(j)$, and a smooth limiting triple
$(A_\infty,\Psi_\infty,\Phi_\infty)$ on $P|_M$ such that
\begin{equation}\label{eq:gauge-convergence}
    g_{i,j}^*(A_i,\Psi_i,\widetilde\Phi_i)
    \longrightarrow(A_\infty,\Psi_\infty,\Phi_\infty)
    \quad\text{in }C^\infty(K_j).
\end{equation}
The limit satisfies \eqref{eq:limit-main}.
\end{proposition}

\begin{proof}
Choose a smooth compact exhaustion with
$K_j\Subset\operatorname{int}K_{j+1}$. Since $\SU(2)$ is
$2$-connected, obstruction theory shows that every principal
$\SU(2)$-bundle over a $3$-manifold is trivializable; in particular, all limiting bundle identifications below
may be made with the fixed bundle $P|_M$. Fix $j$. By
Corollary~\ref{cor:original-scale}, the curvature and all its covariant
derivatives are uniformly bounded on a neighbourhood of $K_j$. Choose a
finite cover by geodesic balls $B_{2r}(x_\alpha)$ so small that
\[
    \sup_{i\geq i(j)}
    \norm{F_{A_i}}_{L^{3/2}(B_{2r}(x_\alpha))}
    <\varepsilon_{\mathrm U},
\]
where $\varepsilon_{\mathrm U}$ is Uhlenbeck's small-curvature constant from
\cite{uhlenbeck1982connections}*{Theorem~1.3}. After applying local gauge
transformations, the connection forms
$a_{i,\alpha}$ satisfy the Coulomb condition and the uniform estimates
\[
    d^*a_{i,\alpha}=0,
    \qquad
    \norm{a_{i,\alpha}}_{W^{1,p}(B_{2r})}
    \leq C_p\norm{F_{A_i}}_{L^p(B_{2r})}
\]
for every finite $p$. The equation
$da_{i,\alpha}=F_{A_i}-a_{i,\alpha}\wedge a_{i,\alpha}$, the Coulomb
condition, and the covariant derivative bounds from
Corollary~\ref{cor:original-scale} bootstrap these estimates to uniform
$C^q$ bounds on $B_r(x_\alpha)$ for every $q$. Thus, after passing to a
subsequence, the local connection forms converge smoothly.

On overlaps, the transition gauges solve the first order equation
\[
    dg_{i,\alpha\beta}
      =g_{i,\alpha\beta}a_{i,\beta}
       -a_{i,\alpha}g_{i,\alpha\beta}.
\]
They consequently have uniform $C^q$ bounds and, after a further
subsequence, converge smoothly to transition functions for a limiting
bundle and connection over $K_j$. Standard Uhlenbeck patching identifies
this limiting bundle, after choosing bundle isomorphisms, with the fixed
bundle $P|_{K_j}$. Equivalently, because principal $\SU(2)$-bundles over
$3$-manifolds are trivializable, one may make all these identifications
inside $P|_{K_j}$. This gives gauge transformations on the fixed bundle
and smooth convergence of $A_i$ on $K_j$.

In these gauges, the identities
\[
    d\Psi_i=\nabla_{A_i}\Psi_i-[a_i,\Psi_i],
    \qquad
    \widetilde\Phi_i=-u_i\Psi_i
\]
combine with Proposition~\ref{prop:transverse-decay} and
Corollary~\ref{cor:smooth-potential} to give uniform $C^q$ bounds and
smooth convergence of $\Psi_i$ and $\widetilde\Phi_i$. Their limits
satisfy $\abs{\Psi_\infty}=1$ and
\[
    \nabla_{A_\infty}\Psi_\infty=0,
    \qquad
    \Phi_\infty=-u\Psi_\infty.
\]
Passing to the limit in the first identity of
\eqref{eq:full-fields-converge} gives
$F_{A_\infty}=-*du\,\Psi_\infty$. It follows also that
\[
    \nabla_{A_\infty}\Phi_\infty=-du\otimes\Psi_\infty,
    \qquad
    F_{A_\infty}=*\nabla_{A_\infty}\Phi_\infty,
\]
so the limit is a reducible monopole.

Repeating the construction along the exhaustion and taking a diagonal
subsequence gives convergence on every $K_j$. The limiting configurations
on $K_j$ and $K_{j+1}$ are related on $K_j$ by a smooth limiting
comparison gauge, obtained as the limit of the comparison gauges for the
sequence. Using these comparison gauges inductively identifies the local
limits with a single smooth limiting configuration on $P|_M$. No
compatibility of the sequence gauges themselves is required. This proves
\eqref{eq:gauge-convergence}.
\end{proof}

\begin{remark}[Flat ambiguity]\label{rmk:flat-ambiguity}
The concentration measure canonically determines the scalar potential
$u$ and hence the scalar longitudinal curvature $-*du$. Once a parallel
$\U(1)$-reduction and its unit generator $\Psi_\infty$ are fixed, the
adjoint-valued curvature is $-*du\,\Psi_\infty$. Two reducible
connections on this fixed reduction with the same curvature differ by a
closed real $1$-form in the parallel direction. Modulo $\U(1)$ gauge
transformations, the remaining freedom is the corresponding flat abelian
holonomy class. A holonomy normalization, or a hypothesis eliminating
this freedom, is therefore necessary for uniqueness of the limiting
connection without passing to subsequences.
\end{remark}

\section{Dirac singularities and charge at infinity}
\label{sec:singularities}

We now identify the singularity at each point of $\cS$. By the local
expansion of the Green function,
\begin{equation}\label{eq:u-local-expansion}
    u=4\pi K_aG(\,\cdot\,,x_a)
      +4\pi\sum_{b\neq a}K_bG(\,\cdot\,,x_b)
      =\frac{K_a}{r_a}+O(1),
\end{equation}
and the differentiated Green expansion gives
\begin{equation}\label{eq:F-local-expansion}
    -*du=K_a*\frac{dr_a}{r_a^2}+O(1).
\end{equation}

\begin{proposition}[Identification of the singular charges]
\label{prop:singular-charges}
For all sufficiently small $r>0$,
\begin{equation}\label{eq:flux-charge}
    \frac1{4\pi}\int_{\partial B_r(x_a)}\ip{F_{A_\infty}}{\Psi_\infty}
    =K_a.
\end{equation}
Thus the abelian reduction of $(A_\infty,\Phi_\infty)$ has a Dirac
singularity of charge $K_a$ at $x_a$.
\end{proposition}

\begin{proof}
The definition of $u$ and $\Delta_xG(x,y)=\delta_y$ give the distributional
identity
\begin{equation}\label{eq:distribution-u}
    \Delta u=4\pi\sum_{b=1}^{\ell}K_b\delta_{x_b}
    \qquad\text{on }X.
\end{equation}
Choose $r$ so that $B_r(x_a)$ contains no other point of $\cS$. Since
$\Delta=d^*d$ is the positive Laplacian, Stokes' theorem applied to
\eqref{eq:distribution-u} gives
\[
    4\pi K_a
    =\int_{B_r(x_a)}\Delta u
    =-\int_{\partial B_r(x_a)}*du.
\]
By Proposition~\ref{prop:gauge-compactness},
$\ip{F_{A_\infty}}{\Psi_\infty}=-*du$, and
\eqref{eq:flux-charge} follows.

Because $\Psi_\infty$ is parallel, it reduces the structure group to its
centralizer $\U(1)$. In the associated rank-two complex vector bundle
$P\times_{\SU(2)}\mathbb C^2$, let $L$ be the eigenline selected by the
sign convention consistent with \eqref{eq:bogomolny}. Under the normalization
$\ip{a}{b}=-2\tr(ab)$, the Chern--Weil degree of $L$ on
$\partial B_r(x_a)$ is precisely the left hand side of
\eqref{eq:flux-charge}. It is therefore $K_a$, which is the defining
charge of the Dirac singularity. The asymptotic expansions
\eqref{eq:u-local-expansion}--\eqref{eq:F-local-expansion} give the usual
Dirac model explicitly.
\end{proof}

\begin{proof}[Proof of Theorem~\ref{thm:main}]
The absence of zeros on compact subsets is Lemma~\ref{lem:zero-free}. Proposition
\ref{prop:potential-convergence} and Corollary
\ref{cor:smooth-potential} prove the scalar convergence in part
\textup{(i)}. Proposition~\ref{prop:transverse-decay} is part
\textup{(ii)}. Proposition~\ref{prop:gauge-compactness} gives the gauges,
the reducible limit, and all equations in part \textup{(iii)}. Finally,
Proposition~\ref{prop:singular-charges} and
\eqref{eq:u-local-expansion}--\eqref{eq:F-local-expansion} prove part
\textup{(iv)}. The final statement concerning the connection follows
from Remark~\ref{rmk:flat-ambiguity}.
\end{proof}

The same Green potential identifies the charge represented by the
limiting field at infinity. For sufficiently large regular values $R$ of
the radius function, set
\[
    X_R:=\{x\in X:\rho(x)<R\},
    \qquad
    \partial X_R=\Sigma_R,
\]
and define
\[
    K_{\mathrm{fin}}:=\sum_{a=1}^{\ell}K_a,
    \qquad
    K_\infty:=k-K_{\mathrm{fin}}\in\mathbb Z_{\geqslant0}.
\]
With $\Sigma_R$ oriented as $\partial X_R$,
\eqref{eq:distribution-u} and
$\ip{F_{A_\infty}}{\Psi_\infty}=-*du$ give, for all sufficiently
large regular $R$,
\begin{equation}\label{eq:residual-charge-infinity}
    \frac1{4\pi}
    \int_{\Sigma_R}\ip{F_{A_\infty}}{\Psi_\infty}
    =K_{\mathrm{fin}}.
\end{equation}
Thus the residual abelian monopole has asymptotic charge
$K_{\mathrm{fin}}$, while $K_\infty$ is the charge escaping through the
AC end. By symmetry of $G$ and \eqref{eq:Green-infinity},
\begin{equation}\label{eq:u-infinity}
    u=\frac{4\pi K_{\mathrm{fin}}}{\Vol(\Sigma)}\rho^{-1}
      +O(\rho^{-1-\eta}).
\end{equation}
For every sufficiently large regular value $R$, the domain $X_R$
contains $\cS$, and hence $\mu(\partial X_R)=0$. Weak convergence
therefore gives
\[
    \lim_{i\to\infty}\mu_i(X_R)
    =\mu(X_R)=4\pi K_{\mathrm{fin}}.
\]
Using \eqref{eq:total-mass-mu}, we obtain
\begin{equation}\label{eq:escaped-charge}
    \lim_{R\to\infty}\lim_{i\to\infty}\mu_i(X\setminus X_R)
    =4\pi K_\infty.
\end{equation}

\begin{corollary}[No mass-renormalized energy loss at infinity]\label{cor:no-loss}
The following are equivalent:
\begin{enumerate}[label=\textnormal{(\roman*)}]
\item $\sum_a K_a=k$;
\item the measures $\mu_i$ are tight;
\item no mass-renormalized energy escapes to infinity;
\item the coefficient of $\rho^{-1}$ in \eqref{eq:u-infinity} equals
the coefficient $4\pi k/\Vol(\Sigma)$ in the asymptotic expansion of
the Higgs defect $m_i-\abs{\Phi_i}$ of the original charge-$k$
monopoles.
\end{enumerate}
\end{corollary}

\begin{proof}
If $K_{\mathrm{fin}}=k$, then the weak limit and each $\mu_i$ have the
same total measure $4\pi k$. Let $R$ be so large that $\cS\subset X_R$.
By weak convergence and $\mu(\partial X_R)=0$,
$\mu_i(X_R)\to4\pi k$; hence $\mu_i(X\setminus X_R)$ is uniformly small for
large $i$, and enlarging $R$ deals with the finitely many remaining
indices. Thus the family is tight. Conversely, tightness and weak
convergence imply convergence of the total measures, so
$4\pi K_{\mathrm{fin}}=\mu(X)=\lim_i\mu_i(X)=4\pi k$. This proves the
equivalence of \textup{(i)} and \textup{(ii)}. Statement \textup{(iii)}
is precisely the vanishing of the right hand side of
\eqref{eq:escaped-charge}, and is therefore equivalent to them.

By \eqref{eq:u-infinity}, the coefficient of $\rho^{-1}$ in the residual
potential is $4\pi K_{\mathrm{fin}}/\Vol(\Sigma)$. On the other hand,
\cite{fadel2023asymptotics}*{Corollary~4.14} gives the coefficient
$4\pi k/\Vol(\Sigma)$ for every original monopole of charge $k$. The two
coefficients agree exactly when $K_{\mathrm{fin}}=k$, proving the
remaining equivalence.
\end{proof}

\begin{remark}
The order of the limits in \eqref{eq:escaped-charge} is essential. The
finite set $\cS$ records concentration at finite points of $X$; it does
not record monopole cores whose centres escape along the AC end.
\end{remark}

\section{Further questions}\label{sec:further-questions}

We conclude with questions concerning uniqueness, charge loss,
compactification of moduli spaces, and extensions of the AC gluing picture.
The examples of \cite{fadel2019limit}*{Section~3} clarify their scope:
on $\mathbb R^3$ they realize both a single Euclidean profile of
arbitrary charge and charge loss through cores escaping to infinity;
the corrected and expanded arXiv version
\cite{fadeloliveira2026limitv5}*{Section~3.5} additionally exhibits
clusters at the same point with several separation scales. Oliveira's construction realizes the stratum of well-separated
charge-one monopoles without charge loss on AC $3$-manifolds with
$b_2(X)=0$.

\begin{problem}[Global uniqueness]
Determine the flat abelian holonomy selected by a large mass sequence and
give geometric or topological hypotheses under which the entire sequence,
rather than a subsequence, converges to the same residual Dirac monopole.
\end{problem}

\begin{problem}[Charge conservation]
The Taubes examples show that the general AC hypotheses do not prevent
charge from escaping to infinity. Find natural restrictions on the
geometry, or intrinsic centring or compactness hypotheses on a large mass
sequence, which force
$\sum_a K_a=k$, or equivalently the tightness of the mass-renormalized
energy measures.
\end{problem}

\begin{problem}[Compactification of moduli spaces]
Construct a compactification of moduli spaces of monopoles of charge $k$
which combines the residual Dirac monopole, the finite clusters at the
mass scale over $\cS$, and the charge escaping through the AC end. Determine the
Euclidean profiles, their relative positions and phases, and the hierarchy
of divergent separation scales in mass-rescaled coordinates. The
intermediate levels of this hierarchy encode relative separations and
should not be interpreted as secondary Euclidean bubbles.
Identify also the data on the AC end needed to record escaping clusters,
and relate the resulting boundary strata to the metric compactifications of
\cite{kottke2015partial}.
\end{problem}

\begin{problem}[Extensions and completeness of the AC gluing picture]
Oliveira's construction \cite{oliveira2016monopoles}*{Theorem~1} realizes
the reverse process when $b_2(X)=0$, the cores are well-separated
charge-one monopoles, and no charge escapes through the end. Extend it to
compatible Dirac backgrounds for $b_2(X)\neq0$ and to prescribed
zero-cycles carrying finite clusters of mass-one monopoles, together with
their relative phases and separation hierarchies, while allowing
additional clusters to escape through the end. Determine the compatibility
among the prescribed finite-point Dirac charges, the asymptotic abelian
charge on the AC end, and the topology of the abelian reduction. On the
stratum of separated charge-one monopoles,
prove an inverse gluing theorem deciding whether every sufficiently large
mass monopole with no charge loss lies in Oliveira's gluing family. Such a
result would give a boundary chart for the
compactification suggested by Theorems~\ref{thm:concentration-clusters}
and~\ref{thm:main}.
\end{problem}

\subsection*{Acknowledgements}

I thank Gon\c{c}alo Oliveira for many important mathematical discussions
that motivated this continuation of our earlier joint work.

This work was supported by the MATH-AmSud project
\emph{Symmetries in Geometry and Physics} (SGP 24-MATH-12) and by CNPq
Universal grant no.~406666/2023-7. It also received support from the
ANR--FAPESP project \emph{BRIDGES---Brazil--France Interplays in Gauge
Theory, Extremal Structures and Stability} (ANR grant ANR-21-CE40-0017;
FAPESP grant 2021/04065-6), and from the University of S\~ao Paulo through
its \emph{Programa de Apoio aos Novos Docentes USP}.

\bibliographystyle{amsra}
\bibliography{references}

\end{document}